\documentclass[12pt]{amsart}
\usepackage{amssymb}
\usepackage{amsfonts}
 \usepackage{amsfonts,amssymb}

\usepackage{mathrsfs}
\let\mathcal\mathscr

\usepackage[all,ps,cmtip]{xy}

\def\llra{\hbox to 10mm{\rightarrowfill}}

\def\lllra{\hbox to 15mm{\rightarrowfill}}

\def\PB{{\widehat B}}

\def\phi{{\varphi}}

\def\cI{\mathcal{I}}

\def\cA{\mathcal{A}}
\def\cF{\mathcal{F}}
\def\cL{\mathcal{L}}
\def\cO{\mathcal{O}}

\def\cP{\mathcal{P}}

\def\cJ{\mathcal{J}}

\def\cQ{\mathcal{Q}}

\def\cV{\mathcal{V}}

\DeclareMathOperator{\codim}{codim}
\DeclareMathOperator{\Pic}{Pic}

\DeclareMathOperator{\vol}{vol}

\newtheorem{lemm}{Lemma}[section]
\newtheorem{theo}[lemm]{Theorem}
\newtheorem{coro}[lemm]{Corollary}

\newtheorem{prop}[lemm]{Proposition}

\newtheorem*{conj*}{Conjecture}

\theoremstyle{definition}

\newtheorem{rema}[lemm]{Remark}

\theoremstyle{remark}
\newtheorem*{remark*}{Remark}
\newtheorem*{note*}{Note}
 \theoremstyle{remark}

\begin{document}
\title[Singularities of divisors]{Singularities of divisors of small degrees on abelian varieties}
\author{Zhi Jiang}
\address{Zhi Jiang, Shanghai Center for Mathematical Sciences, China}
\email{zhijiang@fudan.edu.cn}
\author{Hui Liu}
\address{Hui Liu, Shanghai Center for Mathematical Sciences, China}
\email{22110180027@m.fudan.edu.cn}


%% Mathematical classification (2010)

%% Abstract should be placed before \maketitle (and, in fact, before
%% \begin{document} is best)
\begin{abstract}
Pareschi showed that  an ample divisor of degree $d$ on a simple abelian variety of dimension $g$ has mild singularities when $d<g$. We extend his result to general polarized abelian varieties.
 We also show that ample divisors of degree  $3$ or $4$ on an abelian variety have mild singularities.
\end{abstract}
\maketitle
\section{Introduction}
Given a polarized smooth projective variety $(X, L)$, it is usually important to understand the singularities of $(X, \frac{1}{m}D)$, where $D\in |mL|$ for each $m\geq 1$.

For $D\in |mL|$,   we denote by $\cJ(\frac{1}{m}D)=\mathcal J(X, \frac{1}{m}D)$  the multiplier ideal of $(X, D)$. We also denote by $\cJ^+(\frac{1}{m}D)$  the multiplier ideal $\mathcal J(X, \frac{1-\epsilon}{m}D)$ for $0<\epsilon\ll 1$. It is well-known that $\cJ(\frac{1}{m}D)=\cO_X$ iff $(X, \frac{1}{m}D)$ is log terminal and $\cJ^+(\frac{1}{m}D)=\cO_X$ iff $(X, \frac{1}{m}D)$ is log canonical (see \cite[9.3.9]{Laz2}).  For a reduced divisor $D\in |L|$, we have another measure of singularities, the adjoint ideal $\mathrm{adj}(D)$ defined by Ein and Lazarsfeld. When $\mathrm{adj}(D)$ is trivial, $D$ is normal with rational singularities. Note that this implies that $(X, D)$ is log canonical by adjunction.

When $(A, \Theta)$ is a principally polarized abelian variety, there are   results due to Koll\'ar, Ein-Lazarsfeld, and Hacon describing singularities of $(A, \Theta)$. We summarize these results.
 \begin{itemize}
\item Koll\'ar proved in \cite{Kol2} that $(A, \Theta)$ is log canonical. As a consequence, the closed subset $$\Sigma_k(\Theta):=\{x\in A\mid \mathrm{mult}_x\Theta\geq k\}$$ has codimension $\geq k$ in $A$ for each $k\geq 1$.
\item Ein-Lazarsfeld proved in \cite{EL} that if $\Theta$ is irreducible, it is normal with rational singularities and $(A, \frac{1}{m}D)$ is log canonical for any $D\in |m\Theta|$. Thus $\Sigma_k(\Theta)$ has codimension $k$ for some $k\geq 1$ iff $(A, \Theta)$ splits as a $k$-fold product of PPAVs.
\item Finally, Hacon proved that for any $D\in |m\Theta|$, if $\lfloor \frac{1}{m}D\rfloor=0$, $(A, \frac{1}{m}D)$ is log terminal (see \cite{H1}). Similarly,  $\Sigma_{mk}(D)$ has codimension $k$ and if equality holds,   $(A, \Theta)$ splits as a $k$-fold product of PPAVs.
\end{itemize}

Hacon (\cite{H2}) and Debarre and Hacon (\cite{DH}) later realised that it is possible to extend these results to polarized abelian varieties with small degrees. The main result of \cite{DH} states that, for a polarized abelian variety $(A, L)$ of dimension $>2$ and of degree $2$,  $(A, \frac{1}{m}D)$ is log terminal for any divisor $D\in |mL|$, unless $D=mE$ for some effective divisor $E$. Moreover if   $D\in |L|$ is a prime divisor, $D$ is normal with rational singularities.

Debarre and Hacon  conjectured that similar results hold for $(A, L)$ when $\deg L:=h^0(A, L)<g:=\dim A$. This was partially confirmed by Pareschi in \cite{P}: when $(A, L)$ is a simple polarized abelian variety with $g>\deg L$, every divisor $D\in |L|$ is normal with rational singularities and $(A, \frac{1}{m}D)$ is log terminal for $D\in |mL|$ unless $D=mE$.

The following theorem  refines Pareschi's theorem.

 \begin{theo}\label{general-lc} Let $(A, L)$ be a  polarized abelian variety of dimension $g$ and degree $d$. There exists a  constant $M$ depending only on $g$ and $d$ such that
  \begin{itemize}
  \item[(A)] if $g\geq d$ and there exists an effective boundary $\mathbb Q$-divisor $D\sim_{\mathbb Q}L$ such that $(A, D)$ is not log canonical, there exits an abelian subvariety $K$ of dimension $\geq g-d+1$ such  that $h^0(L|_K)\leq M$;
  \item[(B)] if $g>d$ and there exists an effective $\mathbb Q$-divisor $D\sim_{\mathbb Q}L$ with $\lfloor D\rfloor=0$ such that $(A, D)$ is not Kawamata log terminal or there exists a prime divisor $D\sim L$ such that $D$ does not have rational singularities, there exits an abelian subvariety $K$ of dimension $\geq g-d$ such  that $h^0(L|_K)\leq M$.
  \end{itemize}
 \end{theo}
Indeed, Pareschi's theorem implies  divisors of small degrees in  very general abelian varieties of the corresponding moduli space have mild singularities  and our result shows that the same conclusion holds for general abelian varieties.

We also study divisors of degree $3$ and $4$ in details.  

\begin{theo}Let $(A, L)$ be an indecomposable polarized abelian variety of dimension $g\geq  h^0(A, L)=3$. For any effective boundary $\mathbb Q$-divisor $D\sim_{\mathbb Q}L$, $(A, D)$ is log canonical.

Moreover, if $g>3$ \begin{itemize}\item every prime divisor $D\in |L|$ is normal with rational singularities;
 \item for every $D\in |mL|$, as soon as $\lfloor \frac{1}{m}D\rfloor=0$ and each component of $D$ is ample, $(A, \frac{1}{m}D)$ is Kawamata log terminal.
 \end{itemize}
\end{theo}

\begin{theo} Let $(A, L)$ be an indecomposable polarized abelian variety of dimension $g\geq h^0(L)=4$.

 For any effective boundary $\mathbb Q$-divisor $D\sim_{\mathbb Q}L$ such that every irreducible component of $D$ is ample, $(A, D)$ is log canonical.

\end{theo}
Some results similar  to Theorem 3 and 4 were also obtained  in the thesis of Flavio Blondeau \cite{B}. 
 \subsection*{Acknowledgements}
 The first author thanks Giuseppe Pareschi for 
 The first author is a member of the Key Laboratory of Mathematics for Nonlinear Science, Fudan University and he is supported  by the Foundation for Innovative Research Groups of the Natural Science Foundation of China (No.~12121001), by the National Key Research and Development Program of China (No.~2020YFA0713200), and by the Natural Science Foundation of Shanghai \\ (No.21ZR1404500).

\section{Multiplier ideals, singularities, and adjunction}
We recall the definitions of various singularities, multiplier ideals and some useful facts.
\subsection{Singularities of pairs}
Throughout this paper, we consider pairs $(X, D)$ where $X$ is a smooth projective varieties over an algebraically closed field of characteristic $0$ and  $D=\sum_ic_iD_i$ is an effective $\mathbb Q$-divisors on $X$, where $D_i$ are prime divisors and $c_i\in\mathbb Q$. We say that $D$ is a boundary divisor if $0\leq c_i\leq 1$. We say that two $\mathbb Q$-divisors $D$ and $D'$ are $\mathbb Q$-equivalent if there exists a sufficiently divisible integer $M$ such that $MD\sim MD'$, i.e. $MD$ and $MD'$ are linearly equivalent.

Let $\mu: X'\rightarrow X$ be a log resolution of $(X, D)$ and write $$K_{X'}=\mu^*(K_X+D)+\sum_ia_iE_i.$$ We say that $(X, D)$ is log canonical (resp. Kawamata log terminal) at $x\in X$ iff $a_i\geq -1$ (resp. $a_i>-1$) for all $E_i$ such that $x\in \mu(E_i)$. The multiplier ideal $\cJ(X, D)$ is by definition $\mu_* \cO_{X'} (K_{X'/X}-\lfloor \mu^*D\rfloor)=\mu_* \cO_{X'}(\lceil \sum_ia_iE_i\rceil)$ defining the non-klt locus $\mathrm{Nklt}(X, D)$ of $(X, D)$. One sees that $(X, D)$ is Kawamata log terminal iff $\cJ(X, D)=\cO_X$ and $(X, D)$ is log canonical iff $\cJ(X, (1-\epsilon)D)=\cO_X$ for $0<\epsilon\ll 1$. When $D$ is a prime divisor, we say that $(X, D)$ is  canonical  if $a_i\geq 0$  for all $E_i$ above $X$ which is $\mu$-exceptional.

Given a pair as above, we define the local log canonical threshold at $x\in X$, $$\mathrm{lct}_x(X, D):=\max\{c\in \mathbb Q\mid (A, cD)\; is\; log\; canonical\; at\; x\}$$ and the global log canonical threshold, $$\mathrm{lct}(X, D):=\max\{c\in \mathbb Q\mid (A, cD)\; is\; log\; canonical\}.$$
\subsection{Log canonical centers and adjunctions}
If $(X, D)$ is log canonical at a generic point of $\mu(E_i)$ and $a_i=-1$, we call $\mu(E_i)$ a log canonical center of $(X, D)$. When $(X, D)$ is log canonical at $x$, Kawamata showed that there exists a minimal log canonical center $W$ though $x$, which has rational singularities around $x$ (see \cite{Kaw}). We also have important (sub)adjunctions for log canonical centers. The following version of adjunction for log canonical centers is due to Xiaodong Jiang \cite[Proposition 5.1]{Jx}. %Fujino and Hashizume \cite{FH}. Note that the original theorem of Fujino and Hashizumi deals with $\mathbb R$-pairs and is stated with the language of b-divisors and we only need a simple version.
\begin{theo}\label{adjunction} Let $W$ be a log canonical center of $(X, D)$ and let $\nu: V\rightarrow W$ be the normalization.  Then there exists  an effective $\mathbb Q$-divisor $\mathrm B_V$ on $V$ such that  $$\nu^*(K_X+D)\sim_{\mathbb Q}\mathrm K_V+\mathrm B_V.$$
\end{theo}

\subsection{Adjoint ideals}
When $D$ is a reduced divisor, we can define the adjoint ideal of $(X, D)$. Let $L=\cO_X(D)$ be the corresponding line bundle. We fix  the log resolution $\mu: X'\rightarrow X$ such that the strict transform of $D$ is smooth and write $\mu^*D=D'+F$, where $D'$ is the strict transform of $D$ and $F$ is $\mu$-exceptional. The adjoint ideal $\mathrm{adj}(D)$ is defined to be the ideal sheaf $\mu_* \cO_{X'}(K_{X'/X}-F)$. The adjoint ideal is a measure of the singularities of $D$. Indeed, by \cite[Proposition 9.3.48]{Laz2}, there exists a short exact sequence \begin{multline} \label{EL}0\rightarrow \cO_X(K_X)\xrightarrow{\cdot D}\cO_X(K_X+D)\otimes \mathrm{adj}(D)\\\rightarrow \mu_*\cO_{D'}(K_{D'})\rightarrow 0,\end{multline} and $\mathrm{adj}(D)=\cO_X$ iff $D$ is normal with rational singularities.

%Let $\mu|_{D'}: D'\xrightarrow{\pi} \overline{D}\xrightarrow{n} D$ be the Stein factorization of $\mu|_{D'}$. Thus $n$ is the normalization of $D$. We  have the conductor ideal $\cI_V:=\mathcal Hom(n_*\cO_{\overline{D}}, \cO_D)$ of $D$. Since $D$ is Cohen-Macaulay, $V\subset D$ is of pure codimension $1$ and without embedded points. It is known that the conductor subscheme $V$ is a subscheme of the adjoint subscheme $Z$.

%We apply Grothendieck duality to $\rho:=\mu|_{D'}$:
%\begin{eqnarray*}\rho_{*}\cO_{D'}(K_{D'})=\mathbb R \rho_*\mathcal Hom(\cO_{D'}, K_{D'})\simeq \mathbb R\mathcal Hom(\mathbb R\rho_*\cO_{D'}, L|_D).
%\end{eqnarray*}
%By considering the spectral sequence $$E_2^{i, j}=\mathcal Ext^i(R^j\rho_*\cO_X, L|_D)\Longrightarrow R^{i-j}\rho_*\omega_X$$ and the fact that $\rho_*\cO_X=n_*\cO_{\overline{D}}$,
%we deduce that \begin{eqnarray*}\rho_*\omega_X=L|_D\otimes_{\cO_D} \cI_{Z}\subset L|_D\otimes\cI_V.\end{eqnarray*}

Note that when $D$ is a prime divisor, since $D$ is a divisor of a smooth variety, $D$ has rational singularities iff it has canonical or Kawamata log terminal singularities (see for instance \cite[Theorem 11.1]{Kol1}). Moreover, by \cite[Theorem 7.9]{Kol1}, it is also equivalent to that $(X, D)$ is canonical.

\section{Reducible divisors}
Given a polarized abelian variety $(A, L)$ of dimension $g$ and degree $d$. The polarization type of $L$ can be written as $l=(l_1, l_2, \ldots, l_g)$, where $l_1|l_2\cdots | l_g$ are positive integers. We know that the isomorphism classes of $(A, L)$ forms a coarse  moduli space $\cA_{g, l}$ which is quasi-projective of dimension $\frac{g(g+1)}{2}$. We call $l_g$ the exponent of $L$ or $l$, which is also denoted by $e(L)$.  We know that non-simple  abelian varieties with polarization  type $l$ form countable many irreducible components of $\cA_{g, l}$. Indeed, a generic element of such an irreducible component corresponds to a polarized abelian variety isogeny to the product of two polarized simple abelian varieties $(K, L_K)\times  (B, L_B)$ and the polarization type $L_K$ or $L_B$ is not bounded, i.e. $h^0(L_K)$ or $h^0(L_B)$ could be arbitrary large.

We say that a polarized abelian variety $(A, L)$ is decomposable if $(A, L)$ is isomorphic to a product  $(K, L_K)\times (B, L_B)$ of positive dimensional polarized abelian varieites. When $(A, L)$ is decomposable and the degree $h^0(L)$ is a prime number, any divisor in $|L|$ is reducible and the study of singularities of divisors in $|L|$ can be reduced to lower dimensional case. Thus it is reasonable to assume that $(A, L)$ is indecomposable.

Let $(A, L)$ be an indecomposable polarized abelian variety of dimension $g$ and degree $d$.
 Debarre and Hacon (\cite[Proposition 2]{DH}) showed that if  $g\geq d$ and there exists a reducible divisor $D=D_1+D_2\in |L|$, then $A$ is not simple.  We recall below Debarre and Hacon's observation.

 It is well-known that an effective divisor $D$ on an abelian variety is always nef and the numerical dimension of $D$ is by definition the maximal integer $m$ such that the cohomology class $[D]^m\in H^{2m}(A, \mathbb Z)$ is non-zero.

 Let $g_i$  be the numerical dimension of $D_i$ for $i=1, 2$ and let $K_i$ be the neutral component of the morphism $\varphi_{D_i}: A\rightarrow \Pic^0(A)$ induced by $D_i$. Then $\dim K_i=g-g_i$ and $D_i$ is the  pull-back of an ample divisor on the quotient $p_i: A\rightarrow B_i:=A/K_i$.   Moreover, $D_i|_{K_{3-i}}=D|_{K_{3-i}}$ is ample for $i=1, 2$.  Hence both $(D_1^{g_1}\cdot D_2^{g-g_1})$  and $(D_1^{g-g_2}\cdot D_2^{g_2})$ are positive integers and by Hodge type inequalities, so is $(D_1^{g_1-i}\cdot D_2^{g-g_1+i})$ for each   $0\leq i\leq g_1+g_2-g$.

 We may assume that $g_2\geq g_1\geq 1$. Then $g_1+g_2\geq g$ and one has
 \begin{eqnarray*}
 d&=&\frac{D^g}{g!}\\
 &=&\frac{D_1^{g_1}}{g_1!}\cdot\frac{D_2^{g-g_1}}{(g-g_1)!}+\frac{D_1^{g_1-1}}{(g_1-1)!}\cdot\frac{D_2^{g-g_1+1}}{(g-g_1+1)!}+\cdots+\frac{D_1^{g-g_2}}{(g-g_2)!}\cdot\frac{D_2^{g_2}}{g_2!}\\
 &=&h^0(D_1)h^0(L|_{K_1})+\cdots+h^0(D_2)h^0(L|_{K_2}),
  \end{eqnarray*}
where $\frac{D_1^{g_1-i}}{(g_1-i)!}\cdot\frac{D_2^{g-g_1+i}}{(g-g_1+i)!}$ are positive integers for each $0\leq i\leq g_1+g_2-g$, since they are intersections of integral cohomology classes.

 Thus $g_1+g_2-g+1\leq d\leq g$, we see that $g_1\leq g-1$. Moreover, \begin{eqnarray}\label{equality-reducible}d\geq h^0(L|_{K_1})h^0(D_1)+h^0(L|_{K_2})h^0(D_2)+g_1+g_2-g-1.\end{eqnarray}
\begin{coro}\label{general-comment}Let $(A, L)$ be an indecomposable polarized abelian variety of degree $d=h^0(L)\leq g=\dim A$. If there exists a reducible divisor $D=D_1+D_2\in |L|$, we have $2\leq h^0(L|_{K_1})\leq d$. Moreover, this family of reducible divisors has codimension $$\geq (h^0(L|_{K_1})-1)h^0(D_1)+(h^0(L|_{K_2})-1)h^0(D_2)$$ inside $|L|$.
\end{coro}
 \begin{proof}We note that a general member of  this family of reducible divisors can be written as $D_{1, P}+D_{2, -P}$, where $P\in \Pic^0(B_1)\cap \Pic^0(B_2)\subset \Pic^0(A)$, $D_{1, P}\in |D_1+P|$, and $D_{2 -P}\in |D_2-P|$. We then get the estimation of the codimension from (\ref{equality-reducible}).
 \end{proof}

 \begin{coro}\label{non-lc} Let $(A, L)$ be a polarized abelian variety. Assume that $\dim A=g\geq d=h^0(L)$ and there exists an effective $\mathbb Q$-divisor $D\sim_{\mathbb Q}L$ such that $D$ is not boundary, there exists a proper abelian subvariety $K$ such that $h^0(L|_K)\leq \dim K$.
 \end{coro}
\begin{proof}Let $D_1$ be a component of $D$ whose coefficient is $>1$. Then we may write $L=D_1+(L-D_1)$, where $L_1:=L-D_1\sim_{\mathbb Q}D-D_1$ is an ample divisor. Hence by the computation (\ref{equality-reducible}), $D_1$ is not ample and $h^0(L|_{K_1})\leq d-g_1\leq g-g_1=\dim K_1$.
\end{proof}
 We now classify  explicitly non-prime divisors when $d=3$ or $4$.
\begin{lemm}\label{DH-d=3}\label{DH}Assume that $d=3$ and $|L|$ contains a reducible divisor, then
 \begin{itemize}
 \item[(1)] either there exists a decomposable principally polarized abelian variety $(B, \Theta)$ and an isogeny $\pi: A\rightarrow B$ of degree $3$ such that $D=\pi^*\Theta$;
 \item[(2)] or there exists a principal polarization $\Theta$ on $A$, a quotient $p: A\rightarrow E$ from $A$ to an elliptic curve with connected fibers  such that $D=\Theta+p^*x$, where $x\in E$ is a point, and $\frac{\Theta^{g-1}}{(g-1)!}\cdot p^*x=2$.
\end{itemize}
\end{lemm}

\begin{proof}Let $D=D_1+D_2\in |L|$ be a reducible divisor. We denote by $g_i$ the numerical dimension of $D_i$.
 We have $g_1+g_2\geq g$ and $$3=\frac{D^g}{g!}=\frac{D_1^{g_1}}{g_1!}\cdot\frac{D_2^{g-g_1}}{(g-g_1)!}+\cdots+\frac{D_1^{g-g_2}}{(g-g_2)!}\cdot\frac{D_2^{g_2}}{g_2!}.$$
 We may assume that $g_1<g$ and $g_1\leq g_2$ and assume that $D_i=p_i^*H_i$ where $p_i: A\rightarrow B_i$ is a quotient between abelian varieties with connected fibers and $H_i$ is an ample divisor
 on $B_i$. We denote by $K_i$ the kernel of $p_i$. Note that for any positive dimensional abelian subvariety $K$ of $A$, $h^0(K, L|_K)\geq 2$, otherwise $(A, L)$ would  be decomposable. Then $\frac{D_1^{g_1}}{g_1!}\cdot\frac{D_2^{g-g_1}}{(g-g_1)!)}=h^0(B_1, H_1)\cdot h^0(K, L|_K)\geq 2$.
 Moreover, for any $g-g_2\leq k\leq g_1$, $\frac{D_1^k}{k!}\cdot \frac{D_2^{g-k}}{(g-k)!}\geq 1$.
 Thus we have only two possibilities: $g_1+g_2=g$, both $(B_1, H_1)$ and $(B_2, H_2)$ are PPAVs, and $h^0(K_i, L|_{K_i})=3$ or $g_1=1$, $g_2=g$,  both $(B_1, H_1)$ and $(A, D_2)$ are PPAVs and $h^0(K_1, L|_{K_1})=h^0(K_1, D_2|_{K_1})=2$.
\end{proof}

\begin{lemm}\label{DH-d=4} if $d=4$, then
 \begin{itemize}
 \item[(1)]either there exists a decomposable polarized abelian variety $(A', L')$ of degree $1$ or $2$, an  isogeny $\pi: A\rightarrow B$ of degree $\frac{4}{d}$, such that $D=\pi^*D'$ where $D'\in |L'|$;
 \item[(2)] or there exists a principal polarization $\Theta$ on $A$, a quotient $p: A\rightarrow E$ to an elliptic curve with connected fibers whose kernel is $K$ such that $h^0(K, \cO_A(\Theta)|_K)=3$ and $D=\Theta+p^*x$ for some $x\in E$;
 \item[(3)] or there exists a   polarization $L_2$ on $A$, a quotient $p: A\rightarrow E$ to an elliptic curve with connected fibers whose kernel is $K$ such that $L=L_2\otimes\cO_A(K)$, $h^0(A, L_2)=h^0(K, L_2|_K)=2$  and $D=D'+p^*x$ for some $x\in E$ and $D'\in |L_2\otimes p^*\cO_E(o_E-x)|$;
 \item[(4)]or there exists quotients $p_i: A\rightarrow B_i$ with connected kernels $K_i$ such that $(B_i, \Theta_i)$ a principally polarized abelian variety of dimension $g_i\geq 2$ for $i=1, 2$,  $g_1+g_2=g+1$, $D_i=p_i^*B_i $, and $h^0(K_1, D_2|_{K_1})=h^0(K_2, D_1|_{K_2})=2$.
   \end{itemize}
\end{lemm}
\begin{proof}
The proof is similar to that of Lemma \ref{DH-d=3}. We apply the same notations there. We have \begin{eqnarray*} 4 &=&\frac{D^g}{g!}=\frac{D_1^{g_1}}{g_1!}\cdot\frac{D_2^{g-g_1}}{(g-g_1)!}+\cdots+\frac{D_1^{g-g_2}}{(g-g_2)!}\cdot\frac{D_2^{g_2}}{g_2!}\\
&=&h^0(H_1)h^0(L|_{K_1})+\frac{D_1^{g_1-1}}{(g_1-1)!}\cdot\frac{D_2^{g-g_1+1}}{(g-g_1+1)!}+\cdots+h^0(L|_{K_2})h^0(H_2).
\end{eqnarray*}

If $g_1+g_2=g$, we are in case $(1)$. Note that either $h^0(B_1, H_1)=h^0(B_2, H_2)=1$ and $\deg \pi=4$ or $h^0(B_1, H_1)h^0(B_2, H_2)=2$ and $\deg \pi=2$.

If $g_1+g_2=g+1$, we have two possibilities: either $h^0(B_1, H_1)=1$, $h^0(L|_{K_1})=3$, and $h^0(L|_{K_2})h^0(B_2, H_2)=1$ or $h^0(B_1, H_1)=1$, $h^0(L|_{K_1})=2$, and $h^0(B_2, H_2)h^0(L|_{K_2})=2$.

In the former case, since $h^0(L|_{K_2})=1$ and $(A, L)$ is indecomposable, we have $g_2=g$, $g_1=1$, and $H_2=\Theta$ is a principal polarization on $A$. We are in Case 2.

In the latter case, if $h^0(L|_{K_2})=1$, we have $B_2=A$ and $h^0(A, H_2)=2$. In this case $g_1=1$ and $h^0(H_2|_{K_1})=2$. We are in Case 3.

If we have $h^0(H_2)=1$ in the latter case, we are in Case 4.

Finally, we  need to rule out the possibility that $g_1+g_2\geq 2$. The only non-trivial case we need to rule out is that: $g_1+g_2=g+2$, $ h^0(L|_{K_1})h^0(H_1)=2$, $\frac{D_1^{g_1-1}}{(g_1-1)!}\cdot\frac{D_2^{g-g_1+1}}{(g-g_1+1)!}=1$ and $h^0(L|_{K_2})h^0(H_2)=1$. In this case $g_1=2$, $g_2=g$, $(H_1^2)=2$, $(D_2^{g-2}\cdot K_1)=2(g-2)!$,  $D_1\cdot D_2^{g-1}=(g-1)!$, and $(D_2^g)=g!$. However, by Hodge inequality, we have $(D_1\cdot D_2^{g-1})\geq \sqrt{(D_1^2\cdot D_2^{g-2})(D_2^g)}=2(g-2)!\sqrt{g(g-1)}>2(g-1)!$, which is a contradiction.
\end{proof}
\begin{rema}\label{dimension-reducible}
When $d=3$, the reducible divisors are discrete in Case 1 of Lemma \ref{DH-d=3} and form a family of dimension $1$ in Case 2 of Lemma \ref{DH-d=3}.

When $d=4$, the reducible divisors forms a family of dimension $\leq 2$ and the only $2$-dimensional case is Case 3 of Lemma \ref{DH-d=4}. We claim that in this case $(A, L_2)$ is indecomposable and a general member of this family is the sum of a translate of $K$ with an irreducible ample divisor.

 Let $p: A\rightarrow E:=A/K$ be the quotient.
We consider the short exact sequence $$0\rightarrow L_2(-K)\rightarrow L_2\rightarrow L_2|_K\rightarrow 0.$$ Note that $\chi(L_2(-K))=0$. If $h^0(L_2(-K)\otimes p^*Q)>0$ for $Q\in \Pic^0(E)$ general, $L_2(-K)$ is nef whose Iitaka model dominates $E$. This is impossible, since $L_2(-K)|_K=L|_K$ is ample, thus $L_2(-K)$ is ample which would imply that $\chi(L_2(-K))>0$.  Therefore, the restriction map $H^0(L_2\otimes Q)\rightarrow H^0((L_2\otimes Q)|_K)=H^0(L_2|_K)$ is an isomorphism. Since $(A, L)$ is indecomposable, $(K, L_2|_K)=(K, L|_K)$ is also indecomposable. Thus a general divisor in $|L_2|_K|$ is integral, which implies that a general divisor in $|L_2\otimes p^*Q|$ is also irreducible for $Q\in \Pic^0(E)$ general. In particular, $(A, L_2)$ is indecomposable.
\end{rema}

\section{Sums of subvarieties}

Let $Y$ be an irreducible subvariety of an abelian variety $A$ and let $$I(Y)=\{a\in A\mid a+Y=Y\}.$$
Then the neutral component of $I(Y)_0$ is an abelian subvariety with $\dim I(Y)\leq \dim Y$ and $Y$ is fibred by $I(Y)_0$.
Therefore,  $\dim I(Y)=\dim Y$ if and only if $Y$ is an abelian subvariety. We also know that $\dim I(Y)=0$ iff any desingularization of $Y$ is of general type.

Let $V$ and $W$ be two irreducible subvarieties of an abelian variety $A$ and assume that $\dim V=\dim (V+W)$, then   $W$ is contained in some translate of $I(V)_0$. For example,
if   $\dim W=\dim V=\dim (V+W)$, then both $V$ and $W$ are some translate of an abelian subvariety $K$.

The   following lemma is essentially due to \cite{D}.
\begin{lemm}\label{fibred}
	Let $V$ and $W$ be two irreducible subvarieties of an abelian variety $A$, if  $\dim V +\dim W> \dim (V+W)$, then $V+W$ is fibred by an abelian subvariety of $A$ of dimension $\geq \dim V +\dim W-\dim (V+W)$.
\end{lemm}

\begin{proof}
	Let $m: V\times W \rightarrow A$ be the addition map and let $Z$ be its image. By the hypothesis, for any
	$a\in Z$, we have that $$m^{-1}(a)=\{(x,a-x)|x\in F_a\}$$ has positive dimension, where $F_a\subseteq V$ is the projection of $m^{-1}(a)$ on $W$. Since the differential map of $m$ is just addition of tangent vectors,
	we have $T_a Z\supseteq T_x F_a$ for every $x\in F_a$. Let $\langle F_a\rangle$ be the abelian subvariety
	generated by $F_a$, then $T_a(a+\langle F_a\rangle)\subseteq T_aZ$, since $T_0 \langle F_a\rangle$ is the linear span of
	$T_xF_a, x\in F$. Moreover, since $A$ contains at most countably many abelian subvarieties, $\langle F_a\rangle$ is a fixed abelian subvariety $K$ for $a$ general.
Therefore, by lemma 2.3 of \cite{D}, $Z$ is fibred by $K$.
\end{proof}
\section{Generic vanishing}
We summarize and improve in this section some results from generic vanishing due to Debarre-Hacon, Pareschi, and Pareschi-Popa.

Let $\cF$ be a coherent sheaf on an abelian variety $A$. We define its $i$-th cohomological support loci $$V^i(\cF):=\{P\in\Pic^0(A)\mid H^i(A, \cF\otimes P)\neq 0\}.$$ We say that $\cF$ is IT$^0$, or M-regular, or GV,  if respectively $V^i(\cF)=\emptyset$, or $\codim_{\Pic^0(A)}V^i(\cF)>i$, for $i>0$ or $\codim_{\Pic^0(A)}V^i(\cF)\geq i$, for $i>0$. For a GV-sheaf $\cF$ which is not IT$^0$, we define its gv-index $\mathrm{gv}(\cF)$ to be the non-negative integer $\min\{\codim_{\Pic^0(A)}V^i(\cF)-i\mid i>0\}.$ If $V^i(\cF)=\emptyset$, we set $\codim_{\Pic^0(A)}V^i(\cF)=\infty$. Thus an IT$^0$ sheaf is indeed M-regular.

It may be better to understand IT$^0$, M-regular, and GV sheaves  by considering the Fourier-Mukai transform on $A$. Let $\cP$ be the normalized Poincar\'e bundle on $A\times\Pic^0(A)$ and let $$\Phi_{\cP}: \mathrm{D}(A)\rightarrow\mathrm{ D}(\Pic^0(A))$$ be the Fourier-Mukai functor between the derived category of bounded complexes of coherent sheaves of $A$ and $\Pic^0(A)$. A coherent sheaf $\cF$ on $A$ is IT$^0$ iff $\Phi_{\cP}(\cF)=R^0\Phi_{\cP}(\cF)$ is a vector bundle on $\Pic^0(A)$ and we denote $\widehat{\cF}:=R^0\Phi_{\cP}(\cF)$ in this  case. Let $\cF^{\vee}=R\mathcal{H}om(\cF, \cO_A[\dim A])\in \mathrm{D}(A)$ be the dual of $\cF$. Then $\cF$ be GV (resp. M-regular ) iff $\Phi_{\cP}(\cF^{\vee})=R^0\Phi_{\cP}(\cF^{\vee})$ is a sheaf (resp. torsion-free sheaf) on $\Pic^0(A)$ (see  \cite{PP2}).

Let $(A, L)$ be a   polarized abelian variety of dimension $g$ and degree $d$. Let $D\sim_{\mathbb Q}L$ be a  boundary divisor.  We now consider the ideal sheaf $\cI_Z$ associated with the pair $(A, D)$ where $\cI_Z$ is either $\cJ^+_D:=\cJ(A, (1-\epsilon)D)$ for $0<\epsilon\ll 1$, or $\cA_D:=\mathrm{adj}(D)$ when $D\in |L|$ is a prime divisor, or $\cJ_D:=\cJ(A, D)$ when $\lfloor D\rfloor=0$. Note that $Z$ has codimension $\geq 2$ under these assumptions.

\begin{prop}\label{set-up}Under the above assumptions, $L\otimes \cI_Z$ is IT$^0$ when $\cI_Z=\cJ^+_D$ and $L\otimes \cI_Z$ is  M-regular when $\cI_Z=\cA_D$ or $\cJ_D$. Moreover, in the latter case, $V^i(L\otimes \cI_Z)$ is a union of torsion translates of abelian subvarieties of codimension $\geq i+1$ for $i>0$.
\end{prop}
\begin{proof}
It follows from Nadel's vanishing that $L\otimes \cI_Z$ is IT$^0$ when $\cI_Z=\cJ^+_D$.

When $D\in |L|$ is a prime divisor and $\cI_Z=\cA_D$, we consider the exact sequence (\ref{EL}). Note  that $L$ is IT$^0$ and $\mu_*\cO_{D'}(K_{D'})$ is M-regular since $D$ is a prime divisor and thus not fibred by positive-dimensional abelian subvarieties (see   \cite[Lemma 2.1]{JLT}). Thus $L\otimes\cI_Z$ is also M-regular and $V^i(L\otimes \cI_Z)=V^i(\mu_*\cO_{D'}K_{D'})$ is a union of torsion translates of abelian subvarieties of codimension $\geq i+1$ for $i>0$.

When $\lfloor D\rfloor=0$ and $\cI_Z=\cJ_D$, the statement essentially follows from \cite[The proof of Theorem 7 and 8, case 1]{DH}.  For the reader's convenience, we provide a slightly different argument. Fix a log resolution $\mu: X'\rightarrow A$ of $(A, D)$ and we write $K_{X'}=\mu^*D+\sum_{i}a_iE_i$. Note that $$L\otimes \cI_Z=L\otimes\mu_*\cO_{X'}(K_{X'}-\lfloor \mu^*D\rfloor)=\mu_*(\mu^*L-\lceil \sum_ia_iE_i\rceil).$$ Moreover, $K_{X'}+\mu^*L-\lfloor \mu^*D\rfloor\sim_{\mathbb Q}K_{X'}+\{\mu^*D\}$. The pair $(X, \{\mu^*D\})$ is a klt pair. Thus by the main result of \cite{J} or \cite{M}, $L\otimes\cI_Z$ is either M-regular or there exists a quotient $p: A\rightarrow B$ between abelian varieties and a rank $1$ sheaf $\mathcal L$ on $B$ such that $L\otimes\cI_Z\simeq p^*\cL$. The latter is impossible since $Z\subset A$ has codimension $\geq 2$ by the assumption that $\lfloor D\rfloor=0$ and thus $L\otimes \cI_Z$ is ample restricted to an open subset of $A$ whose complementary has codimension $\geq 2$.
\end{proof}

The following lemma of Debarre and Hacon \cite[Lemma 6]{DH} is crucial.
\begin{lemm}\label{DH-lemma}Under the above assumption, let $\cI_Z$ be either $\cA_D$ or $\cJ_D$. Assume that $V^i(L\otimes \cI_Z) \neq \emptyset$ for some $i>0$, then $\dim Z\geq \dim V^i(L\otimes \cI_Z)+i-1$.
\end{lemm}
The main point of this lemma is that when $\cI_Z$ is either $\cJ^+_D,$ or $\cA_D$, or $\cJ_D$,  $L\otimes\cI_Z$ is, in some sense, always a direct summand of the pushforward of the canonical bundle of some smooth projective variety and hence its higher direct images have nice properties.

The following results are due to Pareschi (see \cite[see the proof of Theorem B]{P}).

\begin{prop}\label{pareschi} Let $\cI_Z$ be an ideal sheaf on $A$. Assume that $L\otimes\cI_Z$ is IT$^0$, $\mathscr{E}xt^k(\cI_Z, \cO_X)=0$ for $k\geq \chi(L\otimes\cI_Z)$. In particular, for each irreducible component $Z_i$ of $Z$ (including the embedding components), $\codim_AZ_i\leq \chi(L\otimes\cI_Z)$.
\end{prop}

\begin{proof}
Indeed, by \cite[the proof of Theorem B (1)]{P}, $\mathscr{E}xt^k(\cI_Z, \cO_X)=0$ for $k\geq \chi(L\otimes\cI_Z)$. Thus $\mathscr{E}xt^k(\cO_Z, \cO_X)=0$ for $k\geq \chi(L\otimes \cI_Z)+1$. For each irreducible component $Z_i$ of $Z$ (including the embedding components), there exists locally a  sheaf $M\hookrightarrow \cO_Z$ whose support is equal to $Z_i$. Let $k$ be the codimension of $Z_i$ inside $A$. We have
$$ \mathscr{E}xt^k(\cO_Z, \cO_A)\rightarrow \mathscr{E}xt^k(M, \cO_A)\rightarrow \mathscr{E}xt^{k+1}(\cO_Z/M, \cO_A).$$

Since the support of $\mathscr{E}xt^{k+1}(M, \cO_A)$ is $Z_i$ and the support of $$\mathscr{E}xt^{k+1}(\cO_Z/M, \cO_A)$$ has codimension $\geq k+1$, we conclude that $\mathscr{E}xt^k(\cO_Z, \cO_A)\neq 0$. Thus, for each irreducible component $Z_i$ of $Z$ (including the embedding components), $\codim_AZ_i\leq \chi(L\otimes\cI_Z)$.
\end{proof}

\begin{prop}\label{codim} Let $\cI_Z$ be either $\cA_D$ or $\cJ_D$. Assume that  $V^i(L\otimes \cI_Z) \neq \emptyset$ for some $i>0$, $\codim_A Z\leq \mathrm{gv}(L\otimes \cI_Z)+1\leq \chi(L\otimes\cI_Z)+1$.\end{prop}

\begin{proof}By Lemma \ref{DH-lemma}, for $i>0$ such that $\codim_AV^i(L\otimes \cI_Z)-i$ computes the gv-index of $L\otimes \cI_Z$, we have $\codim_AZ\leq g-\dim V^i(L\otimes \cI_Z)-i+1=\mathrm{gv}(L\otimes \cI_Z)+1\leq \chi(L\otimes\cI_Z)+1$, where the last inequality follows from the main result of \cite{PP1}.
\end{proof}

\begin{lemm}\label{not1}Assume that $(A, L)$ is indecomposable. Let $\cI_Z$ be an ideal sheaf such that $L\otimes\cI_Z$ is M-regular and $Z\neq \emptyset$. Then  $2\leq  \chi(A, L|_Z)\leq d-1$, unless $Z$ is a reduced point.
\end{lemm}
\begin{proof}By Proposition \ref{set-up}, we know that $L\otimes\cI_Z$ is M-regular and hence $\chi(L\otimes \cI_Z)>0$ by the main result of \cite{PP1}. By the short exact sequence $$0\rightarrow L\otimes \cI_Z\rightarrow L\rightarrow L|_Z\rightarrow 0,$$ we see that $L|_Z$ is M-regular and its gv-index is $\geq 2$. If $L|_Z$ is not IT$^0$, we know by \cite{PP1} that $\chi(A, L|Z)\geq \mathrm{gv}(L|_Z)\geq 2$. If $L|_Z$ is IT$^0$ and $\chi(A, L|_Z)=1$,  the Fourier-Mukai transform $\widehat{(L|_Z)}$ of $L|_Z$ is a line bundle on $\Pic^0(A)$. Thus $Z$ is a translate of an abelian subvariety $K$ of $A$ and $(K, L|_K)$ is a principally polarized abelian variety. If $\dim K>0$, $(A, L)$ is decomposable by \cite[Lemma 1]{DH}, which contradicts the assumption.
\end{proof}

%Shibata extended some results of generic vanishing to log canonical pairs in \cite{S}.
%\begin{theo}\label{lc-gv}Let $(X, D)$ be a log canonical pair. Given a morphism $f: X\rightarrow A$ from $X$ to an abelian variety and a Cartier divisor $H$ on $X$ such that $H\sim_{\mathbb Q}K_X+D$, $R^if_*\cO_X(H)$ is GV for any $i\geq 0$. Moreover, $V^0(f_*\cO_X(H)))$ is a union of torsion translates of abelian subvarieties of $\Pic^0(A)$.\end{theo}\begin{proof}See \cite[Theorem 3.5 and Proposition 4.2]{S}.\end{proof}

 \begin{prop}\label{components}
 Let $(A, L)$ be a polarized abelian variety. Assume that  $(A, D)$ is a log canonical pair. Let $C$ be the union of some log canonical centers of $(A, D)$. Then,
  \begin{itemize}
   \item when $D\sim_{\mathbb Q}cL$  and $0<c<1$,  $L\otimes \cI_C$ is IT$^0$;
   \item when  $D\sim_{\mathbb Q}L$,  $L\otimes \cI_C$ is GV.
  \end{itemize}
 \end{prop}
 \begin{proof}Fix a log resolution $\rho: X\rightarrow A$ of $(A, D)$. We write
 \begin{eqnarray}\label{discrepancy}K_X+E_1+\cdots+E_n+E^{<1}=\rho^*D,\end{eqnarray} where $-\lfloor E^{<1}\rfloor $ is an effective $\mu$-exceptional divisor. We may assume that $E_1,\ldots, E_k$ are the divisors over irreducible component of $C$ with discrepancy $-1$. Note that $L\otimes\cI_C=\rho_*(\rho^*L(-E_1-\cdots-E_k-\lfloor E^{<1}\rfloor)$.

 In the first case, $$\rho^*L-E_1-\cdots-E_k-\lfloor E^{<1}\rfloor\sim_{\mathbb Q}K_X+E_{k+1}+\cdots+E_n+\{E^{<1}\}+(1-c)\rho^*L,$$ we then conclude by \cite[Theorem 3.16.3]{F} that $L\otimes\cI_C$ is IT$^0$.

 In the second case, by the same argument as above, we have $\pi_M^*(L\otimes\cI_C)\otimes L$ is IT$^0$, where $\pi_M: A\rightarrow A$ is the multiplication-by-M map and thus by Hacon's theorem (see the main theorem of \cite{H3} or \cite[Theorem 5.2]{JP}), $L\otimes\cI_C$ is GV.
 \end{proof}

\begin{coro}\label{g>d} Let $(A, L)$ be an indecomposable polarized abelian of dimension $g$ and degree $d$.
 \begin{itemize}
\item If $g\geq d$ and  $\cJ^+_D=\cI_Z$ for some $Z\neq \emptyset$, $2\leq \chi(A, L|_V)\leq d-2$.
\item If $g>d$ and   $\cA_D$ or $\cJ_D$ is $\cI_Z$ for some $Z\neq \emptyset$,  $2\leq \chi(A, L|_Z)\leq d-1$.
 \end{itemize}

\end{coro}
\begin{proof}In the first case, since $L\otimes\cI_V$ is IT$^0$ and $V$ has codimension $\geq 2$, we have $2\leq \chi(L\otimes\cI_V)\leq d-1$. If $\chi(L\otimes\cI_V)=d-1$, by Lemma \ref{not1},  the only possibility is that $V$ is a reduced point and we get a contradiction by Proposition \ref{pareschi}. In the second case, when $Z$ is a reduced point, $\dim V^1(L\otimes\cI_Z)=\dim \mathrm{Bs}|L|\geq g-d$ and we get a contradiction by Lemma \ref{DH-lemma}.
\end{proof}
%We have a slight extension of Proposition \ref{codim}.\begin{lemm}Let $(A, L)$ be a polarized abelian variety. Assume that $D\sim_{\mathbb Q}L$ is a boundary divisor and $(A, D)$ is a log canonical pair. For any minimal lc center $V$ of $(A, D)$, we have $$\dim V\geq g-\chi(L\otimes\cI_V)-1.$$\end{lemm}\begin{proof}We apply a similar argument as the one in Lemma \ref{DH-lemma}. We already know that $L\otimes\cI_V$ is a GV sheaf and its cohomological support loci is a union of torsion translates of abelian subvarieties of $\Pic^0(A)$. Let $-Q+\PB$ be an irreducible component of $V^i(L\otimes \cI_V)$.\end{proof}
 \section{General polarized abelian varieties $(A, L)$ with $\dim A\geq \deg L$}
The goal of this section is to prove the following result.
 \begin{theo}\label{general-lc} Let $(A, L)$ be a  polarized abelian variety of dimension $g$ and degree $d$. There exists a  constant $M$ depending only on $g$ and $d$ such that
  \begin{itemize}
  \item[(A)] if $g\geq d$ and there exists an effective boundary $\mathbb Q$-divisor $D\sim_{\mathbb Q}L$ such that $(A, D)$ is not log canonical, there exits an abelian subvariety $K$ of dimension $\geq g-d+1$ such  that $h^0(L|_K)\leq M$;
  \item[(B)] if $g>d$ and there exists an effective $\mathbb Q$-divisor $D\sim_{\mathbb Q}L$ with $\lfloor D\rfloor=0$ such that $(A, D)$ is not Kawamata log terminal or there exists a prime divisor $D\sim L$ such that $D$ does not have rational singularities, there exits an abelian subvariety $K$ of dimension $\geq g-d$ such  that $h^0(L|_K)\leq M$.
  \end{itemize}
 \end{theo}

 We  remark that $(A)$ is an extension of Corollary \ref{non-lc}. By the following lemma,  $(B)$ implies that Pareschi's theorem \cite[Theorem 1]{P} holds for general polarized abelian subvarieties in the moduli space of $(A, L)$.

\begin{lemm}\label{exponent}
Let $(A, L)$ be a polarized abelian variety. Let $K$ be an abelian subvariety and let $B$ be the Poincar\'e complement of $K$ with respect to $L$. Then $e(L|_B)$ divides $e(L)e(L|_K)$.
\end{lemm}
 \begin{proof}
 We first fix an isogeny $\pi: A\rightarrow A'$ such that there exists a principal polarization $L'$ on $A'$ and $L=\pi^*L'$. Let $K'=\pi(K)$. Then it is easy to verify that $B':=\pi(B)$ is the Poincar\'e complement of $K'$ with respect to $L'$. By \cite[Corollary 12.1.2]{BL}, $e(L'|_{K'})=e(L'|_{B'})$. Since $e(L'|_{K'})$ divides $e(L|_K)$ and $e(L|_B)$ divides $e(L)e(L'|_{B'})$, we conclude the proof.
 \end{proof}
 The proofs of $(A)$ and $(B)$ are almost identical. Let's first prove $(A)$ and point out the necessary modifications of arguments to prove $(B)$.

 We also remark that the constant $M$ is explicit but our computation in the proof is certainly not optimal. As we can see in the next sections, the restriction should be much more strict.
 \subsection{The proof of $(A)$}
 \begin{proof}

  Let $0<c=\mathrm{lct}(A, D)<1$, $(A, cD)$ is log canonical. Let $Z$ be a minimal log canonical center of $(A, cD)$. We know that $Z$ is normal with at worst rational singularities (see \cite{Kaw}).
 After a small perturbation of $D$ (see \cite[Proposition 8.7.1]{Kol4}), we may assume that $\cI_Z=\cJ(cD)$.

   Moreover, by Proposition \ref{components}, $L\otimes \cI_Z$ is IT$^0$. We denote by $\chi=\chi(L\otimes\cI_Z)$. Note that $2\leq \chi \leq d-2$ and $\dim Z\geq g-\chi\geq 2$ by Proposition \ref{pareschi} and Corollary \ref{g>d}.

We now apply the argument of Debarre and Hacon in \cite{DH}.
  Let $J\subset A\times |L|$ be the incident variety $$\{(a, s)\in A\times |L|\mid s|_{a+Z}=0\}.$$ Note that the projection $p: J\rightarrow A$ is surjective and $J$ is indeed a $\mathbb P^{\chi-1}$-bundle over $A$.
  Let $q: J\rightarrow |L|$ be the natural projection. For $s\in q(J)$ general, we denote by $J_s$ the fiber over $s$ and $D_s$ the divisor corresponding to $s$. By definition, \begin{eqnarray}\label{j}J_s=\{a\in A\mid a+Z\subset D_s\}.\end{eqnarray}

Let $J_s^i$ be an irreducible component of $J_s$.
  We consider the addition morphism:
  $$\mu: J_s^i\times Z\rightarrow V_s:=J_s^i+Z\subset D_s\subset A.$$

  Since $s\in q(J)$ general, $\dim J_s^i=\dim J_s\geq g+\chi-d$ and $\dim (J_s^i+Z)\leq \dim D_s\leq g-1$, a general fiber of $\mu: J_s^i\times Z\rightarrow V_s$ is of dimension $\geq g-d+1$. Let $t\in V_s$ be general and denote by $F_t$ the corresponding fiber of $\mu$. We denote by $Z_t$ the image of the projection from $F_t$ to $Z$.  Observe that $$F_t=\{(t-z, z)\mid z\in Z_t\}.$$ We then conclude by Lemma \ref{fibred} that $Z_t$ generates an abelian subvariety $K$ of dimension $\geq g-d+1\geq 2$ and $V_s$ is fibred by $K$, i.e. $K+V_s=V_s$.

We remark that $F_t$ or $Z_t$ may be reducible. After taking the normalizations $\overline{J_s^i}$ of $J_s^i$ and $\overline{J_s^i\times Z}$ of $J_S^i\times Z$ and considering   the Stein factorization
\begin{eqnarray*}\xymatrix{
\overline{J_s^i}\times Z\ar[dr]\ar[r] &\overline{J_s^i+Z}\ar[r] &J_s^i+Z\\
& J_s^i\times Z,\ar[ur]&&}
 \end{eqnarray*}we see that the normalization of each irreducible component of $F_t$ deforms to each other and hence each irreducible component of $Z_t$ generates the same abelian variety $K$.

   Considering the quotient $\pi: A\rightarrow A/K$. We denote respectively $\hat{J_s^i}$, $\hat{Z}$, and $\hat{V_s}$ the images of $J_s^i$, $Z$, and $V_s$ in $A/K$. By considering the addition morphism $\hat{\mu}$ on $A/K$, we have the commutative diagram:
  \begin{eqnarray*}
  \xymatrix{
  J_s^i\times Z\ar[r]^{\mu} \ar[d]_{\pi|_{J_s^i}\times \pi|_Z} & V_s\ar[d]\ar@{^{(}->}[r]  & A\ar[d]^{\pi} \\
  \hat{J_s^i}\times \hat{Z} \ar[r]^{\hat{\mu}} & \hat{V_s}\ar@{^{(}->}[r]& A/K.}
  \end{eqnarray*}

  By (\ref{j}), $J_s^i=\pi^{-1}(\hat{J_s^i})$. Moreover, since $F_t$ is a general fiber of $\mu$, $\hat{\mu}$ is generically finite. We then conclude that $\hat{Z}$ is not fibred by positive-dimensional tori of $A/K$.
  Indeed, if $\hat{Z}$ is fibred by a positive-dimensional abelian tori $B$, so is $\hat{V_s}$. Then $V_s$ is fibred by $\pi^{-1}(B)$ and so is $J_s^i$. Then both $\hat{J_s^i}$ and $\hat{Z}$ are fibred by $B$, which implies that $\hat{\mu}$ is not generically finite.

 Let $n=\dim Z$.  We  claim that $(L^{n}\cdot Z)\leq (L^g)=g!d$. Indeed, since $L\otimes \cI_Z$ is IT$^0$, it is continuously globally generated. Thus, $$\bigcap_{P\in \Pic^0(A)}\mathrm{Bs}(|L\otimes \cI_Z\otimes P|)=Z.$$ Inductively, there exist divisors $D_i\equiv L$, for $1\leq i\leq g-n$ such that $D_1\cap D_2\cdots\cap D_{g-n}$ is of pure dimension $n$ and contains $Z$ as an irreducible component. Hence $$(L^n\cdot Z)\leq (L^g)=g!d.$$

 We claim that $(L^{\dim K}\cdot K)$ is bounded by a constant depending only on $g$ and $d$. We have the following commutative diagram
 \begin{eqnarray*}
 \xymatrix{
 Z\ar@{^{(}->}[r]\ar[d]^{p_Z}&A\ar[d]
\\
  \hat{Z}\ar@{^{(}->}[r]& A/K.}
 \end{eqnarray*}
 By Theorem \ref{adjunction}, $$L|_{Z}\sim_{\mathbb Q}K_{Z}+(\text{effective}\; \mathbb Q\text{-divisors}).$$ By Viehweg's weakly positivity (see \cite{V}), $K_{Z}/K_{\hat{Z}}$ is pseudo-effective \footnote{In order to apply Viehweg's theorem, we need  that $Z$ and $\hat{Z}$ are smooth. However, it is easy to verify that it is harmless to assume that both varieties are smooth in our arguments.}. Since $\hat{Z}$ is not fibred by abelian subvarieties, it is of general type. By Fujita's approximation (see \cite[Theorem 11.4.4]{Laz2}), there exists an $\mathbb Q$-ample divisor $H$ on $\hat{Z}$ such that $K_{\hat{Z}}-H$ is $\mathbb Q$-effective and $0<\vol(K_{\hat{Z}})-(H^{\dim \hat{Z}})\ll 1$. Then $L|_{Z}-p_Z^*H$ is also pseudo-effective. Note that a general fiber of $p_Z$ is exactly $Z_t$. We then have
 \begin{eqnarray*}
 (L|_{Z})^n > ((L|_{Z})^{\dim Z_t}\cdot (p_Z^*H)^{\dim \hat{Z}})= (H^{\dim \hat{Z}})(L|_K^{\dim Z_t}\cdot Z_t).
 \end{eqnarray*}
 Thus $$g!d\geq (L^n\cdot Z)\geq \vol(K_{\hat{Z}})(L|_K^{\dim Z_t}\cdot Z_t).$$
 We also know that $\vol(K_{\hat{Z}})\geq  2(\dim \hat{Z})!$ by Severi's inequality (see \cite{Bar} and \cite{Zh}). Therefore, $$(L|_K^{\dim Z_t}\cdot Z_t)\leq \frac{g!d}{2(\dim \hat{Z})!}.$$

Note that an irreducible component of $Z_t$ generates $K$ and $2L|_K$ is globally generated. For $H_1,\ldots, H_{\dim Z_t-1}\in |2L|_K|$ general, an irreducible component $C$ of $Z_t\cap H_1\cdots\cap H_{\dim Z_t-1}$ is an irreducible curve generating  $K$. By \cite[Proposition 4.1]{D3} $$\big(\frac{(L|_K\cdot C)}{\dim K}\big)^{\dim K}\geq h^0(K, L|_K).$$ Since $$(L|_K\cdot C)\leq \frac{2^{\dim Z_t-1}g!d}{(\dim \hat{Z})!},$$ $h^0(L|_K)$ is bounded by a constant depending only on $g$ and $d$.
\end{proof}
\subsection{The proof of $(B)$}
\begin{proof}
By $(A)$, we may assume that $(A, D)$ is log canonical.  We denote by $\cI_Z$ the multiplier ideal sheaf $\cJ(D)$ when $D\sim_{\mathbb Q} L$ is a $\mathbb Q$-divisor such that $\lfloor D\rfloor=0$ and $(A, D)$ is not Kawamata log terminal or the adjoint ideal sheaf $\mathrm{adj}(A, D)$ when $D\in |L|$ is a prime divisor, which  does not have canonical singularities.

  In either case, $Z$ is reduced and each irreducible component is a log canonical center of $(A, D)$.
By Proposition \ref{set-up}, $L\otimes\cI_Z$ is M-regular and by Corollary \ref{g>d}, $1\leq \chi(L\otimes\cI_Z)\leq d-2$.   Let $Z_1$ be an irreducible component of $Z$ of maximal dimension. We have that $L\otimes\cI_{Z_1}$ is GV by Proposition  \ref{components}. Moreover, since $L\otimes \cI_Z\subset L\otimes\cI_{Z_1}\subset L$,
$$1\leq \chi(L\otimes\cI_Z)\leq \chi:= \chi(L\otimes\cI_{Z_1}).$$
We also know by Proposition \ref{codim} that $\dim Z_1=\dim Z\geq g-\chi(L\otimes \cI_Z)-1\geq g-\chi-1$.

We then consider the subscheme $$\cJ:=\{(a, s)\in A\times |L| \mid s|_{a+Z_1}=0\}.$$ Let $J$ be the unique component of $\cJ$ dominating $A$ with reduced scheme structure. We denote respectively by $p$ and $q$ the projections from $J$ to $A$ and $|L|$. We also know that $\dim J=g+\chi-1$.

For $a\in A$ general, we then denote by $J_s^i$ an irreducible component of a fiber of $J\rightarrow q(J)$ over $s\in q(J)$ such that $a\in J_s^i$. Then $\dim J_s^i\geq g+\chi-d$ and $$J_s^i\times Z\rightarrow V_s:=J_s^i+Z\subset D_s,$$ where $D_s\subset A$ is the corresponding divisor.  Since $a\in A$ general, we know that $J_s^i$ is an irreducible component of $$\{z\in A\mid z+Z\subset D_s\}$$ through $a$.

Then we apply exactly the same argument of the proof of $(A)$. One difference is that $L\otimes\cI_{Z_1}$ may not be M-regular and thus may not be continuously globally generated. However, since $L\otimes \cI_Z$ is M-regular, there still  exist divisors $D_i\equiv L$, for $1\leq i\leq g-\dim Z_1$ such that $D_1\cap D_2\cdots\cap D_{g-\dim Z_1}$ is of pure dimension $\dim Z_1$ and contains $Z_1$ as an irreducible component. Another difference is that $Z_1$ may not be normal. We thus need to consider the normalization of $Z_1$ while applying Theorem \ref{adjunction}.
  \end{proof}
\section{Divisors of degree $3$}
\begin{theo}\label{lc}Let $(A, L)$ be an indecomposable polarized abelian variety of degree $d=3$ of dimension $g\geq 3$. Fix an effective $\mathbb Q$-divisor $D\sim_{\mathbb Q}L$, if the pair $(A, D)$ is not log canonical,   there exits a fibration $p: A\rightarrow E$ to an elliptic curve such that $D=a_1p^*x+D_2$, where $x\in E$ is a point, $a_1>1$, and $D_2$ is an effective $\mathbb Q$-divisor.
\end{theo}
\begin{proof}
We first claim that for any boundary $\mathbb Q$-divisor $D\sim_{\mathbb Q} L$, $\cJ_D^+=\cO_A$ and hence $(A, D)$ is log canonical. Let $\cI_Z=\cJ^+_D$. If $\cI_Z\neq \cO_X$, $Z$ is a subscheme of codimension $\geq 2$. However, $L\otimes\cI_Z$ is IT$^0$ by Proposition \ref{set-up} and by Lemma \ref{g>d}, $2\leq \chi(L|_Z)\leq 1$, which is impossible.

 We now classify the $\mathbb Q$-divisors $D\sim_{\mathbb Q}L $ which are not boundary divisors.
We write $D=a_1D_1+D_2$, where $D_1$ is a prime divisor, $a_1>1$, and $D_2$ is an effective $\mathbb Q$-divisor. Then $L-D_1\sim_{\mathbb Q}(a_1-1)D_1+D_2$ is ample. Since an ample divisor on an abelian variety is always linearly equivalent to an effective divisor, there exists an integral reducible divisor $D'=D_1+D_2'\sim L$ and $D_2'$ is ample.
Thus by Lemma \ref{DH}, we are in the second cases.

\end{proof}

\begin{theo}\label{klt-d=3}Let $(A, L)$ be an indecomposable polarized abelian variety of dimension $g>h^0(A, L)=3$. Then
 \begin{itemize}\item a prime divisor $D\in |L|$ is normal and has rational singularities;
 \item for an effective $\mathbb Q$-divisor $D\sim_{\mathbb Q}L$ such that $\lfloor D\rfloor=0$, if $(A,  D)$ is not Kawamata log terminal, there exists an abelian subvariety $K$ of $A$ of dimension $g-2$ such that $\chi(L|_K)=2$ and $\cJ_{D}$ is the ideal sheaf of a translate of $K$. Moreover, if each irreducible component of $D$ is ample, $(A, D)$ is Kawamata log terminal.
 \end{itemize}

\end{theo}

\begin{proof}We argue by contradiction. Assume the contrary, $(A, D)$ is not purely log terminal (or canonical) in the first case and $(A, D)$ is not Kawamata log terminal in the second case.
Let $\cI_Z\subsetneqq \cO_A$ be the adjoint ideal $\cA_D$ in the first case and be the multiplier ideal $\cJ(D)$ in the second case.
 Moreover, in any case, we already know that the pair $(A, D)$ is log canonical by Theorem \ref{lc}, hence the subscheme $Z$ is reduced.

By Corollary \ref{g>d}, $\chi(A, L\otimes \cI_Z)=1$ and $\chi(L|_Z)=2$ and by Proposition \ref{codim} , $\dim Z=g-2$. Note that $Z$ is connected. Indeed, if $Z$ has $2$ connected components $Z_1$ and $Z_2$, since $L|_{Z_i}$ have gv-index $\geq 2$,  $\chi(L|_{Z_1})=\chi(L|_{Z_2})=1$ and $Z=Z_1\sqcup Z_2$. By Lemma \ref{not1}, both $Z_1$ and $Z_2$ are reduced points and this is a contradiction.

  Let $\cJ$ be the subvariety of $A\times |L|$ defined by the closed condition $$\{(a, s)\in A\times |L|\mid s|_{a+Z}=0\}.$$ Let $p$ and $q$ be respectively the projection from $\cJ$ to $A$ and $|L|$. For each $a\in A$, $p^{-1}(a)$ is identified with the subspace $$\mathbb PH^0(A, L\otimes \cI_Z\otimes P_{\varphi(a)}),$$ where $\varphi: A\rightarrow \Pic^0(A)$ is the isogeny induced by $L$ and $P_{\varphi(a)}$ is the corresponding line bundle on $A$. Let $J$ be the unique irreducible component of $\cJ$ which dominates $A$ via $p$. We will still use $p$ and $q$ to denote respectively the projection from $J$ to $A$ and $|L|$. Since $\chi(A, L\otimes \cI_Z)=h^0(A, L\otimes \cI_Z\otimes P)=1$ for $P\in\Pic^0(A)$ general, $p: J\rightarrow A$ is birational.

We then consider $q: J\rightarrow |L|=\mathbb P^2$. It is clear that $\dim q(J)\geq 1$.

 We claim that $q$ is surjective.

 Assume that $\dim q(J)=1$. Let $s\in q(J) $ be a general point and let $J_s$ be the corresponding fiber $J_s$. By definition, we see that $J_s+Z\subset D_s$, where $D_s$ is the divisor corresponding to $s$. Since $\dim Z=g-2>0$ and $\dim J_s=g-1$, we conclude that $D_s$ contains a component which is not ample by Lemma \ref{fibred}. Hence $|L|$ contains a one-dimensional family of reducible divisors. We are in case $(2)$ of Lemma \ref{DH}, namely there exists a quotient $p_E: A\rightarrow E$ with connected kernel $K$ such that $D_s=p_E^*{x_s}+\Theta_s$, where $x_s\in E$ is a point and $\Theta_s$ is a principal polarization.  Thus $q(J)\subset |L|$ is an elliptic curve. Since $J$ is birational to $A$, we  conclude that each connected component of $J_s$ is a translate of a fixed abelian variety $K$.  We apply the analysis of Lemma \ref{DH} to conclude that $L(-K)$ is a principal polarization. Since $Z$ is connected,  after a translation, we may assume that  $Z$ is contained in $K$.
 Then we have $L(-K) \subset L\otimes \cI_Z$. Since $L(-K)$ is ample, $L\otimes \cI_Z$ is M-regular, and
  $\chi(A, L(-K))=\chi(A, L\otimes \cI_Z)=1$, $L(-K)=L\otimes \cI_Z$, which is absurd.

Therefore we may assume that $q: J\rightarrow |L|=\mathbb P^2$ is surjective. For $s\in |L|$ general, $J_s=q^{-1}(s)$ is a subvariety of $A$ of pure dimension $g-2$, the corresponding divisor $D_s$ is a prime divisor. We note that $J_s+Z\subset D_s$. Since $\dim J_s+\dim Z>\dim D_s$ and $D_s$ is ample,  $J_s+Z$ is a proper subset of $D_s$ by \ref{fibred}. Therefore $\dim J_s=\dim Z=\dim (J_s+Z)$. Hence each connected component of $J_s$ is a translate of a fixed abelian subvariety $K$ of dimension $g-2$. Since $Z$ is reduced, after a translation, we may assume that  $Z$ is $K$. We have $h^0(K, L|_K)=2$.

We now show that any prime divisor $D\in |L|$ is normal and has rational singularities.

We fix a \'etale cover $\pi: A\rightarrow A'$ of degree $3$ from $A$ to a PPAV $(A', \Theta)$ such that $L\simeq \pi^*L'$, where $L':=\cO_{A'}(\Theta)$. Since $h^0(K, L|_K)=2$, $\pi|_K$ induces an isomorphism from $K$ to its image, which we still denote by $K$. Let $B$ and $B'$ be respectively the Poincar\'e complementary of $K$ in $A$ and $A'$.
It is easy to see that $\pi|_{B}$ induces a \'etale cover $B\rightarrow B'$ of degree $3$.
We have the following  commutative diagram
\begin{eqnarray*}
\xymatrix{
K\times B\ar[rr]^{\pi|_K\times\pi|_{B}}\ar[d]^{\mu}&& K\times B'\ar[d]^{\mu'}\\
A\ar[rr]^{\pi} && A'.}
\end{eqnarray*}
Note that both $\mu$ and $\mu'$ are $(\mathbb Z_2\times \mathbb Z_2)$-covers. We also have $(K\times B, \mu^*L)\simeq (K\times B, L|_K\boxtimes L|_B)$ (see \cite[Corollary 5.3.6]{BL}).
Since $K$ is a log canonical center of $(A, D)$, $K\times b_i$ also a log canonical center of $(K\times B, \mu^*D)$, where $b_i\in K\cap B$ for $1\leq i\leq 4$. Note that $D$ dominates $K$. For two general points $x, y\in K$, let $D_1$ and $D_2$ be the restriction of $D$ in $x\times B$ and $y\times B$. Then $D_1, D_2\in |L|_B|$ are two divisors without common component. Moreover, both $(B, D_1)$ and $(B, D_2)$ have $b_i$ as a log canonical center (see for instance \cite[Theorem 9.5.35]{Laz2}). Thus their local intersection multiplicity $(D_1\cdot D_2)_{b_i}\geq 4$ for $1\leq i\leq 4$ (see for instance \cite[Corollary 6.46]{Kol3}). Then $$12=(L_B^2)\geq \sum_{i=1}^4(D_1\cdot D_2)_{b_i}\geq 16,$$ which is a contradiction.

When each component of $D\sim_{\mathbb Q} L$ is ample (we just need that each component of $D$ is not a pull-back of a divisor on $A/K$), the same argument works.
\end{proof}
\section{Divisors of degree $4$}
 \begin{theo}\label{lc-d=4} Let $(A, L)$ be a polarized abelian variety of dimension $g\geq h^0(L)=4$.
 Let $D\sim_{\mathbb Q}L$ be an effective boundary divisor. Assume that $(A, D)$ is not log canonical, then $\cJ^{+}_D$ defines a translate of an abelian variety $K$ of dimension $g-2$ such that $\chi(L|_K)=2$.
\end{theo}
\begin{proof}
 Let $\cI_Z:=\cJ^+_{D}$. Then by assumption, $Z$ is a subscheme of codimension $\geq 2$. By Nadel's vanishing, $L\otimes\cI_Z$ is IT$^0$. Since $(A, L)$ is indecomposable, by Corollary \ref{g>d}, $1\leq \chi(L\otimes\cI_Z)\leq 2$. Moreover, since $Z$ is of codimension $\geq 2$, $\chi(L\otimes\cI_Z)=2$. Hence both $L\otimes\cI_Z$ and $L|_Z$ are IT$^0$ sheaves whose holomorphic Euler characteristics are $2$. By Proposition \ref{pareschi}, we know that $Z$ is of pure dimension $g-2$. Note that $Z$ is connected, otherwise $(A, L)$ is decomposable.

We also claim that $Z$ is reduced. Indeed, let $c=\mathrm{lct}(A, D)<1$, $(A, cD)$ is a log canonical pair and $\cJ^+_D\subset \cJ(A, cD)$. Similarly, $L\otimes \cJ(A, cD)$ is IT$^0$ and $\chi(L\otimes \cJ(A, cD))=2$. Thus $\cJ^+_D=\cJ(A, cD)$ defines a reduced subscheme.

We consider again the subvariety $$J:=\{(a, s)\mid s\mid_{a+Z}=0\}\subset A\times |L|.$$ Note that the first projection $p: J\rightarrow A$ makes $J$ a $\mathbb P^1$-bundle over $A$. It is easy to see that $J=\mathbb P_A(\cV)$, where $\cV$ is the rank $2$ vector bundle $$\varphi_L^*R^0p_{2*}(p_1^*(L\otimes\cI_Z)\otimes \cP),$$ where $\cP$ is the Poincar\'e line bundle on $A\times \Pic^0(A)$, $p_1$ and $p_2$ are the natural projections on $A\times\Pic^0(A)$, and $\varphi_L: A\rightarrow \Pic^0(A)$ is the isogeny induced by $L$.

It is clear that the fibers of $J\rightarrow A$ vary inside $|L|$ and hence $\dim q(J)\geq 2$. We claim that the second projection $q: J\rightarrow |L|\simeq \mathbb P^3$ is surjective. Otherwise, $\dim q(J)=2$ and we see as before that for any $s\in q(J)$, the corresponding divisor $D_s$ is reducible. We have a $2$-dimensional reducible divisors inside $|L|$. Thus by Lemma \ref{DH-d=4} and Remark \ref{dimension-reducible}, we are in Case (3) of Lemma \ref{DH-d=4}. The subvariety $q(J)\subset |L|$ corresponds to the family of divisors
$$\{p^*x+|L-p^*x|\mid\; \text{where}\; x\in E \;\text{and}\; L-p^*x\equiv L_2 \}.$$

Thus for $s\in q(J)$ general, a connected component of the fiber $J_s$ of $J$ over $s$ is a translate  of  the kernel $K$ of $A\rightarrow E$. We also see that $Z$ is contained in a translate of $K$. We may assume that $Z$ is a subscheme of $K$.  Then we have $L(-K) \subset L\otimes \cI_Z$, but $L(-K)$ is ample, $L\otimes \cI_Z$ is M-regular,and $\chi(A, L(-K))=\chi(A, L\otimes \cI_Z)=2$,
we must have $L(-K)=L\otimes \cI_Z$, which is again impossible.% We write $D=aK+D_2$, where no component of $D_2$ coincides with $K$ and $0\leq a\leq 1$. Note that $D_2|_K\sim_{\mathbb Q}L|_K$. Moreover, $h^0(L|_K)=2$ and $(K, L|_K)$ is indecomposable, thus $(K, D_2|_K)$ is log canonical and by inversion of adjunction, $(A, D)$ is also log canonical around $K$, which is a contradiction.

Thus we may assume that $q: J\rightarrow |L|$ is surjective. For $s\in |L|$ general, let $J_s$ be the corresponding fiber of $q$. We have $\dim J_s=g-2$. Note that $J_s+Z\subset D_s$ and $D_s$ is irreducible. Thus $\dim (J_s+Z)=\dim Z=\dim J_s=g-2$. Hence there exists an abelian subvariety $K$ of $A$ of dimension $g-2$ such that $Z$ is  supported on a translate of $K$ and $J_s$ is a union of translates of $K$. After a translation, we may assume that $Z=K$.

%We then show that $Z$ is reduced and hence $Z=K$. Assume the contrary, the kernel  $\cM$ of $\cO_Z\rightarrow \cO_K$ is nonzero.  By Lemma \ref{non-reduced}, $\cM$ contains a subsheaf $F^{m-1}\cM$ which is a quotient of $\cO_K^{\oplus M}$ and $F^{m-1}\cM$ is torsion-free, since $Z$ is pure. Therefore, there exists an injective map $\cO_K\rightarrow F^{m-1}\cM\hookrightarrow \cM$. Note that $L|_Z$ is IT$^0$ and $\chi(L|_Z)=2$. On the other hand, $h^0(L|_K)=\chi(L|_K)\geq 2$ because $(A, L)$ is indecomposable. Thus for any $Q\in \Pic^0(A)$, $H^0(\cM\otimes L\otimes Q)=H^0(L|_Z\otimes Q)=2$ and the restriction map $$r_Z: H^0(L|_Z\otimes Q)\rightarrow H^0(L|_K\otimes Q)$$ is zero. However, $H^0(A, L\otimes Q)\rightarrow H^0(K, L|_K\otimes Q)$ factors through $r_Z$. This shows that the restriction map $H^0(A, L\otimes Q)\rightarrow H^0(K, L|_K\otimes Q)$ is always zero, which is absurd.

\end{proof}

\begin{coro}\label{lc-d=4andample} Let $(A, L)$ be a polarized abelian variety of dimension $g\geq h^0(L)=4$.
 Let $D\sim_{\mathbb Q}L$ be an effective boundary divisor. Assume that each component of $D$ is ample, $(A, D)$ is log canonical.
\end{coro}
\begin{proof}Assume the contrary, then, after translation, we may assume that $\cJ_D^+=\cI_K$ by Theorem \ref{lc-d=4}.
By considering the Poincar\'e duality, there exists an abelian surface $B\subset A$ such that the addition map $\mu: B\times K\rightarrow A$ is an isogeny, whose kernel is isomorphic to $G:=K\cap B$ and $h^0(L|_B)=2|G|$. Note that $(B\times K, \mu^*D)$ is not log canonical and $\cJ(\mu^*D)^+=\cI_{K\times G}$.
If each component of $D$ is ample,  for two general points $t_1$ and $t_2\in K$, $D_i:=D|_{t_i\times B}$, $i=1, 2$ are two effective $\mathbb Q$-divisors on $B$ without common component and each points in $G\subset B$ is a non-lc center of $(B, D_i)$. We have again  their local intersection multiplicity $(D_1\cdot D_2)_{b}> 4$ for $b\in G$ (see for instance \cite[Corollary 6.46]{Kol3}). Thus  $$4|G|=(D_1\cdot D_2)\geq \sum_{b\in G}(D_1\cdot D_2)_{b}> 4|G|,$$ which is a contradiction.
\end{proof}

\end{document}